\documentclass[12pt]{amsart}
\usepackage{tikz-cd}
\usetikzlibrary{matrix,arrows,decorations.pathmorphing}
\usepackage{amssymb}
\usepackage{amsmath,amscd}
\usepackage{mathrsfs}
\usepackage{url}
\usepackage{cite}
\usepackage{fullpage}
\usepackage{hyperref}
\usepackage{verbatim}
\usepackage{comment}
\usepackage{fancyvrb}
\usepackage{fvextra}
\usepackage{mathrsfs}
\newtheorem{thm}{Theorem}[section]
\newtheorem{prop}[thm]{Proposition}
\newtheorem{lem}[thm]{Lemma}
\newtheorem{cor}[thm]{Corollary}

\usepackage{wasysym}
\usepackage{graphicx}
\usepackage{wrapfig}

\theoremstyle{definition}
\newtheorem{definition}[thm]{Definition}
\newtheorem{example}[thm]{Example}
\newtheorem{rem}[thm]{Remark}

\numberwithin{equation}{section}

\newcommand{\M}{\mathscr{M}}
\newcommand{\X}{\mathcal{X}}

\newcommand{\V}{\mathscr{V}}

\newcommand{\B}{\mathcal{B}}

\newcommand{\zz}{\mathbb{Z}}
\newcommand{\qq}{\mathbb{Q}}

\newcommand{\Cp}{\mathscr{C}}
\newcommand{\Pp}{\mathscr{P}}

\newcommand{\p}{\mathbb{P}}
\newcommand{\pp}{\mathbb{P}}

\renewcommand{\H}{\mathcal{H}}
\newcommand{\Hp}{\mathscr{H}}

\renewcommand{\P}{\mathcal{P}}
\newcommand{\E}{\mathcal{E}}
\renewcommand{\O}{\mathcal{O}}

\newcommand{\Mg}{\M_g}
\newcommand{\Ub}{\mathcal{V}}
\renewcommand{\gg}{\mathbb{G}}
\renewcommand{\tilde}{\widetilde}

\DeclareMathOperator{\Id}{Id}

\DeclareMathOperator{\ch}{ch}

\DeclareMathOperator{\GL}{GL}
\DeclareMathOperator{\SL}{SL}
\DeclareMathOperator{\Pic}{Pic}

\DeclareMathOperator{\Sym}{Sym}

\DeclareMathOperator{\PGL}{PGL}
\DeclareMathOperator{\BSL}{BSL}
\DeclareMathOperator{\BPGL}{BPGL}
\DeclareMathOperator{\BGL}{BGL}
\DeclareMathOperator{\gen}{gen}

\begin{document}
\title{The integral Picard groups of low-degree Hurwitz spaces}

\author{Samir Canning, Hannah Larson}
\thanks{During the preparation of this article, S.C. was partially supported by NSF RTG grant DMS-1502651. H.L. was supported by the Hertz Foundation and NSF GRFP under grant DGE-1656518. This work will be part of S.C.'s and H.L.'s Ph.D. theses.}
\email{srcannin@ucsd.edu}
\email{hlarson@stanford.edu}
\subjclass[2010]{14C15, 14C17}
\maketitle
\begin{abstract}
We compute the Picard groups with integral coefficients of the Hurwitz stacks parametrizing degree $4$ and $5$ covers of $\p^1$. As a consequence, we also determine the integral Picard groups of the Hurwitz stacks parametrizing simply branched covers. For simple branching, the Picard groups are finite, with order depending on the genus.
\end{abstract}
\section{Introduction}
Let $\Hp_{k,g}$ be the Hurwitz stack parametrizing degree $k$ covers by genus $g$ curves of $\p^1$, up to automorphisms of the target $\p^1$. Let $\Hp^{s}_{k,g} \subseteq \Hp_{k,g}$ be the open substack parametrizing simply branched covers. The Hurwitz space Picard rank conjecture posits that $\Pic(\Hp^{s}_{k,g})\otimes \qq=0$. 
For $k \geq 4$, the complement of $\mathscr{H}_{k,g}^s \subset \mathscr{H}_{k,g}$ consists of two irreducible divisors: $D$, parametrizing covers with two points of ramification in the same fiber (pictured left) and $T$, parametrizing covers with a point of triple ramification (pictured right).

\begin{figure}[h!] \label{TD}
\includegraphics[width=4.75in]{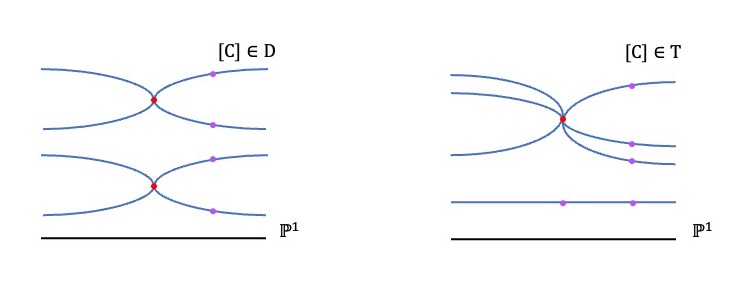}
\caption{Components of the complement of $\Hp_{k,g}^s \subset \Hp_{k,g}$}
\end{figure}

In \cite{DP}, Deopurkar--Patel prove that for $k \geq 4$, the classes of $T$ and $D$ are linearly independent in $\Pic(\Hp_{k,g})\otimes \qq$. Note that $D$ is empty when $k = 3$.
For $k \geq 4$, the Picard rank conjecture is then equivalent to $\Pic(\mathscr{H}_{k,g}) \otimes \qq \cong \qq^{\oplus 2}$.
The conjecture is known by work of Stankova-Frenkel and Deopurkar--Patel for $k\leq 5$ \cite{Stankova-Frenkel,DP}, and for $k>g-1$ by work of Mullane \cite{Mu}.

In this paper, we will study $\Pic(\Hp_{k,g})$ and $\Pic(\Hp_{k,g}^s)$ with \emph{integral} coefficients. Much less is known in this case: Arsie--Vistoli \cite{AV} computed $\Pic(\Hp_{2,g})$, and Bolognesi--Vistoli \cite{BV} computed $\Pic(\Hp_{3,g})$. In both cases, there are torsion classes depending on the genus. Our main theorem shows that this torsion phenomenon does not extend to $k=4,5$ when $g\geq 3$. 

\begin{thm}\label{Pic}
For $g \geq 2$, the integral Picard groups of the Hurwitz stacks are as follows.
\begin{enumerate}
    \item We have
    \[
    \Pic(\Hp_{4,g})=\begin{cases} \zz \oplus \zz/10\zz & \text{if $g = 2$} \\ \zz\oplus \zz & \text{if $g \geq 3$.} \end{cases}
    \]
    \item We have 
    \[
    \Pic(\Hp_{5,g})=\begin{cases} \zz \oplus \zz/10\zz &\text{if $g = 2$} \\ \zz\oplus \zz & \text{if $g \geq 3$.}\end{cases}
    \]
\end{enumerate}
\end{thm}

We also provide explicit line bundles generating $\Pic(\Hp_{k,g})$ (Sections \ref{gb3}, \ref{gb4}, and \ref{gb5}). One of those line bundles is the determinant of the Hodge bundle, which is pulled back from the moduli space of curves. The torsion in the case $g=2$ arises from the torsion in the Picard group of the moduli space of genus $2$ curves, as the first Chern class $\lambda$ of the Hodge bundle is $10$-torsion when $g=2$ \cite{V3}.

\begin{rem}
This $10$-torsion phenomenon is also present when $k = 3$, namely we shall prove $\Pic(\Hp_{3,2}) = \zz/10\zz$, correcting the $g = 2$ case of \cite{BV}.
\end{rem}

The classes of the divisors $T$ and $D$ have been previously computed \cite{DP,part2}. Using these computations, we determine the integral Picard groups of the simply branched Hurwitz spaces. They are, of course, torsion.
\begin{cor} \label{mc}
For $g \geq 2$, the integral Picard groups of the simply branched Hurwitz stacks are as follows:
\begin{enumerate}
    \item We have
    \[ \Pic(\Hp_{3,g}^s) =
    \begin{cases} \zz/2\zz & \text{if $g = 2$} \\ \zz/(4g + 6)\zz \oplus \zz/3\zz & \text{if $g \geq 3$ odd} \\
    \zz/(8g + 12)\zz \oplus \zz/3\zz  & \text{if $g \geq 3$ even} 
    \end{cases} \]
    \item We have
    \[\Pic(\Hp_{4,g}^s) = \begin{cases} \zz/18\zz \oplus \zz/2\zz &\text{if $g = 2$} \\ \zz/(8g+20)\zz \oplus \zz/12\zz & \text{if $g \geq 3$ odd} \\
    \zz/(4g+10)\zz \oplus \zz/12\zz & \text{if $g \geq 3$ even.} 
    \end{cases} \]
    \item We have
    \[\Pic(\Hp_{5,g}^s) = \begin{cases}
    \zz/44\zz \oplus \zz/2\zz &\text{if $g = 2$} \\
    \zz/(4g + 14)\zz \oplus \zz/12\zz & \text{if $g \geq 3$ odd} \\
    \zz/(8g+28)\zz \oplus \zz/12\zz & \text{if $g \geq 3$ even.}
    \end{cases}\]
\end{enumerate}
\end{cor}

The paper is structured as follows. In Section \ref{hurwitzstacks}, we define the stacks $\Hp_{k,g}$ and their closely related counterparts $\H_{k,g}$. 
When $k \leq 5$, these stacks have a nice relationship with stacks parametrizing pairs of vector bundles on $\pp^1$-fibrations, respectively pairs of vector bundles on $\pp^1$-bundles (the former may not have a relative degree $1$ line bundle; this distinction is important for results with integral coefficients). We construct these stacks of vector bundles on $\p^1$-fibrations, respectively $\pp^1$-bundles, in Section \ref{p1bundles}, compute their integral Picard groups and explain how they are related to each other. 
In Section \ref{trigonal}, we calculate $\Pic(\Hp_{3,g})$, which is originally due to Bolognesi--Vistoli, in a different way. This new perspective is used to prove Corollary \ref{mc}(1).
Along the way, we also obtain results in genus $2$ that will be useful in Sections \ref{tetragonal} and \ref{pentagonal}.
In Section \ref{tetragonal}, we prove Theorem \ref{Pic}(1) and from it Corollary \ref{mc}(2). 
In Section \ref{pentagonal}, we prove Theorem \ref{Pic}(2) and Corollary \ref{mc}(3).

\subsection*{Acknowledgments} We are grateful to our advisors, Elham Izadi and Ravi Vakil, respectively, for the many helpful conversations. We are grateful to Anand Deopurkar and Anand Patel for enlightening discussions about their work \cite{DP,DP2}. We thank Andrea Di Lorenzo for his comments and suggestions.

\section{Hurwitz Stacks}\label{hurwitzstacks}
We say a morphism $P\rightarrow S$ is a $\pp^1$-fibration if it is a flat, proper, finitely presented morphism of schemes whose geometric fibers are isomorphic to $\p^1$. We define the unparametrized Hurwitz stack $\Hp_{k,g}$ of degree $k$, genus $g$ covers of $\pp^1$ to be the stack whose objects over a scheme $S$ are of the form $(C\rightarrow P\rightarrow S)$ where $P\rightarrow S$ is a $\pp^1$-fibration, $C\rightarrow P$ is a finite, flat, finitely presented morphism of constant degree $k$, and the composition $C\rightarrow S$ is smooth with geometrically connected fibers. We do not impose the condition that our covers $C\rightarrow \p^1$ be simply branched. In the case $k=3$, $\Hp_{3,g}$ is the stack $\mathcal{T}_g$ from \cite{BV}.

The parametrized Hurwitz scheme $\Hp_{k,g}^{\dagger}$ is defined similarly, except $P\rightarrow S$ is replaced by $\p^1_S$. Therefore, the unparametrized Hurwitz stack is the $\PGL_2$ quotient of the parametrized Hurwitz scheme.
%The scheme $\Hp_{k,g}^\dagger$ is smooth [add REF].
There is also a natural action of $\SL_2$ on $\Hp_{k,g}^\dagger$ (via $\SL_2 \subset \GL_2 \rightarrow \PGL_2$).
\begin{center}
    \textit{We shall use script font $\Hp_{k,g} := [\Hp_{k,g}^\dagger/\PGL_2]$ for the $\PGL_2$ quotient,  \\
    and caligraphic font $\H_{k,g} := [\Hp_{k,g}^\dagger/\SL_2]$ for the $\SL_2$ quotient.}
\end{center}

Explicitly, the $\SL_2$ quotient $\H_{k,g}$ is the stack whose objects over a scheme $S$ are families $(C\rightarrow P\rightarrow S)$ where $P = \pp V \rightarrow S$ is the projectivization of a rank $2$ vector bundle $V$ with trivial determinant, $C\rightarrow P$ is a finite flat finitely presented morphism of constant degree $k$, and the composition $C\rightarrow S$ has smooth fibers of genus $g$.
The benefit of working with $\H_{k,g}$ is that the $\SL_2$ quotient is equipped with a universal $\pp^1$-bundle $\P \to \H_{k, g}$ that has a relative degree one line bundle $\O_{\P}(1)$ (a $\pp^1$-fibration does not).

The Hurwitz stacks come with universal diagrams 
\begin{equation}\label{univdiagram}
\begin{tikzcd}
\Cp \arrow{r}{\alpha} \arrow{rd}[swap]{f} & \Pp \arrow{d}{\pi} \\
& \Hp_{k,g},
\end{tikzcd}
\end{equation}
where $\Cp\rightarrow \Hp_{k,g}$ is the universal curve, $\Cp\rightarrow \Pp$ is the universal degree $k$ cover, and $\Pp\rightarrow \Hp_{k,g}$ is the universal $\p^1$-fibration. One can also form the analogous diagram for $\H_{k,g}$. We set 
\[
\lambda:=c_1(f_*\omega_f),
\]
which is pulled back from the moduli space of curves $\Mg$.
\section{Stacks of vector bundles on $\p^1$}\label{p1bundles}
In this section, we discuss these stacks of vector bundles on $\p^1$-fibrations and $\p^1$-bundles, and compute their Picard groups. 

\begin{definition}\label{Stacks}
Let $r,d$ be nonnegative integers. 
\begin{enumerate}
    \item The objects of $\V_{r,d}$ are pairs $(P\rightarrow S, E)$ where $P\rightarrow S$ is a $\pp^1$-fibration over a $k$-scheme $S$ and $E$ is a locally free sheaf of rank $r$ on $P$ whose restriction to each of the fibers of $P\rightarrow S$ is globally generated of degree $d$. A morphism between objects $(P\rightarrow S, E)$ and $(P'\rightarrow S', E')$ is a Cartesian diagram
    \[
    \begin{tikzcd}
P' \arrow[d] \arrow[r, "F"] & P \arrow[d] \\
S' \arrow[r]                & S          
\end{tikzcd}
    \]
    together with an isomorphism $\phi:F^*E\rightarrow E'$.
    \item The objects of $\Ub_{r,d}$ are triples $(S, V, E)$ where $S$ is a $k$-scheme, $V$ is a rank $2$ vector bundle on $S$ with trivial determinant, and $E$ is a rank $r$ vector bundle on $\pp V$ whose restrictions to the fibers of $\pp V \rightarrow S$  are globally generated of degree $d$. A morphism between objects $(S,V,E)$ and $(S',V',E')$ is a Cartesian diagram
     \[
    \begin{tikzcd}
\pp V' \arrow[d] \arrow[r, "F"] & \pp V \arrow[d] \\
S' \arrow[r]                & S          
\end{tikzcd}
    \]
    together with an isomorphism $\phi:F^*E\rightarrow E'$.
\end{enumerate}
\end{definition}
\noindent

\par Bolognesi--Vistoli \cite{BV} gave a presentation for $\V_{r,d}$ as a quotient stack, which we briefly summarize here. Let $M_{r,d}$ be the affine space that represents the functor which sends a scheme $S$ to the set of matrices of size $(r+d)\times d$ with entries in $H^0(\p^1_S,\mathcal{O}_{\pp^1_S}(1))$. We can identify such a matrix with the associated map
\[
\mathcal{O}_{\p^1_S}(-1)^{d}\rightarrow \mathcal{O}_{\p^1_S}^{r+d}.
\]
Let $\Omega_{r,d} \subset M_{r,d}$ denote the open subscheme parametrizing injective maps with locally free cokernel. The group $\GL_d$ acts on $M_{r,d}$ by multiplication on the left, $\GL_{r+d}$ acts by multiplication on the right, and $\GL_2$ acts by change of coordinates on $H^0(\pp^1_S, \O_{\pp^1_S}(1))$. These actions commute with each other and leave $\Omega_{r,d}$ invariant, and hence $\GL_{d}\times \GL_{r+d} \times \GL_2$ acts on $\Omega_{r,d}$. There is a copy of $\mathbb{G}_m$ inside of $\GL_{d}\times \GL_{r+d}\times \GL_2$ embedded by $t\mapsto(t \Id_d, \Id_{r+d},t^{-1}\Id_2)$. The image $T$ acts trivially on $M_{r,d}$ and so we can define an action of the quotient 
\[
\Gamma_{r,d}:=\GL_{d}\times \GL_{r+d}\times \GL_2/T
\]
on $\Omega_{r,d}$. There is an exact sequence
\[
1\rightarrow \GL_{d}\times \GL_{r+d}\rightarrow \Gamma_{r,d}\rightarrow \PGL_2\rightarrow 1, 
\]
where the map $\Gamma_{r,d}\rightarrow \PGL_2$ is induced by the projection of $\GL_{d}\times \GL_{r+d}\times \GL_2\rightarrow \GL_2$.
\begin{thm}[Bolognesi--Vistoli \cite{BV}, Theorem 4.4]\label{BVthm}
There is an isomorphism of fibered categories
\[
\V_{r,d}\cong[\Omega_{r,d}/\Gamma_{r,d}.]
\]
\end{thm}

A slight modification of the argument in Bolognesi--Vistoli  gives a quotient structure for $\Ub_{r,d}$, which we have utilized in our previous work \cite{part1, part2}.

\begin{prop} \label{Vprime}
There is an isomorphism of fibered categories
\[
\Ub_{r,d}\cong[\Omega_{r,d}/\GL_{d}\times \GL_{r+d}\times \SL_2].
\]
\end{prop}
\begin{proof}
The proof is the same as in \cite[Theorem 4.4]{BV}, except that instead of taking $P\rightarrow S$ a $\pp^1$-fibration in the definition of the various stacks, we take $P=\p V\rightarrow S$ where $V$ is a rank $2$ vector bundle with trivial determinant.
\end{proof}

To parametrize the linear algebraic data associated to a low degree cover of $\pp^1$, we will need to construct stacks parametrizing pairs of vector bundles on $\pp^1$. 
These stacks are products of the form $\Ub_{r,d} \times_{\BSL_2} \Ub_{s,e}$, which parametrize a pair of vector bundles on the same $\pp^1$-bundle, or $\V_{r,d}\times_{\BPGL_2} \V_{s,e}$, which parametrize a pair of vector bundles on the same $\pp^1$-fibration. 

Let $G_{r,d,s,e} := \GL_d \times \GL_{r+d} \times \GL_e \times \GL_{s+e}$.
The group $G_{r,d,s,e} \times \SL_2$ acts on $M_{r,d}$ via the projection $G_{r,d,s,e} \times \SL_2 \to \GL_d \times \GL_{r+d} \times \SL_2$; and similarly on $M_{s,e}$ via the projection $G_{r,d,s,e} \times \SL_2 \to \GL_e \times \GL_{s+e} \times \SL_2$.
By Proposition \ref{Vprime}, we have
\begin{equation} \label{asq}
\Ub_{r,d} \times_{\BSL_2} \Ub_{s,e} = [\Omega_{r,d} \times \Omega_{s,e}/G_{r,d,s,e} \times \SL_2].
\end{equation}

Let $T_d$ and $T_{r+d}$ denote the universal vector bundles on $\BGL_d$ and $\BGL_{r+d}$; similarly, let $S_e$ and $S_{s+e}$ be the universal vector bundles on $\BGL_{e}$ and $\BGL_{s+e}$. 
The integral Chow ring of $\mathrm{B}(G_{r,d,s,e} \times \SL_2)$ is the free $\zz$-algebra on the Chern classes of $T_d, T_{r+d}, S_e, S_{s+e}$, together with the universal second Chern class $c_2$ on $\BSL_2$.
Let us denote these classes by
\begin{align*}
t_i &= c_i(T_d) \qquad \text{and} \qquad u_i = c_i(T_{r+d}) \\
v_i &= c_i(S_e) \qquad \text{and} \qquad  w_i = c_i(S_{s+e}).
\end{align*}
Since $\Omega_{r,d} \times \Omega_{s,e}$ is open inside the affine space $M_{r,d} \times M_{s,e}$, 
the excision and homotopy properties imply 
\begin{equation} \label{gens}
\text{$\Pic(\Ub_{r,d} \times_{\BSL_2} \Ub_{s,e})$ is generated by the restrictions of $t_1, u_1, v_1, w_1$.}
\end{equation}

We now identify the restrictions of the tautological bundles $T_d$ and $T_{d+r}$ in terms of the universal rank $r$, degree $d$ vector bundle on $\pp^1$.
Let $\pi: \mathcal{P} \rightarrow \Ub_{r,d}$ be the universal $\pp^1$-bundle. We write $z := c_1(\O_{\mathcal{P}}(1)) \in A^1(\mathcal{P})$.
We have $c_2 = c_2(\pi_* \O_{\mathcal{P}}(1)) \in A^2(\Ub_{r,d})$, the universal second Chern class, pulled back via the natural map $\Ub_{r,d} \rightarrow \BSL_2$).
Note that $c_1(\pi_* \O_{\mathcal{P}}(1)) = 0$, so by the projective bundle formula
\[A^*(\mathcal{P}) = A^*(\Ub_{r,d})[z]/(z^2 + \pi^*c_2).\]
Let $\E$ be the universal rank $r$, degree $d$ vector bundle on $\P$.
 The Chern classes of $\E$ may thus be written as
\[c_i(\E) = \pi^* a_i + (\pi^* a_i') z \qquad \text{where} \qquad a_i \in A^i(\Ub_{r,d}), \quad a_i' \in A^{i-1}(\Ub_{r,d}).\]
Note that $a_1' = d$. Let $\gamma: \Ub_{r,d} \times_{\BSL_2} \Ub_{s,e} \to B(G_{r,d,s,e} \times \SL_2)$ be the structure map. Then by \cite[Lemma 3.2]{Brd} (noting that $\det(\pi_* \O_{\P}(1))$ is trivial), we have 
\begin{equation} \label{idT}
\gamma^* T_d = \pi_* \E(-1) \qquad \text{and} \qquad \gamma^*T_{r+d} = \pi_* \E.
\end{equation}
Since $R^1\pi_*\E(-1)$ and $R^1\pi_*\E$ are zero, Grothendieck--Riemann--Roch says that
the Chern characters of $\pi_*\E(-1)$ and $\pi_*\E$ are push forwards by $\pi$ of polynomials in the $c_i(\E)$ and $z$. The push forward of such a polynomial is a polynomial in the $a_i, a_i'$ and $c_2$. In particular, the restrictions of $t_1$ and $u_1$ to $\Pic(\Ub_{r,d})$ are linear combinations of $a_1$ and  $a_2'$. We calculate this explicitly in the following example. %(See Example \ref{exc1}.)
%\begin{lem} \label{res}
%Let $\gamma: \Ub_{r,d} \rightarrow \BGL_d \times \BGL_{r+d}$ be the natural map. 
%Let $T_d$ and $T_{r+d}$ respectively denote the universal rank $d$ and $r+d$ vector bundles on $\BGL_d \times \BGL_{r+d}$.
%We have
%\[\gamma^* T_d = \pi_* \E(-1) \qquad \text{and} \qquad %\gamma^*T_{r+d} = \pi_* \E.\]
%In particular, the restrictions of $t_i$ and $u_i$ to $A^*(\Ub_{r,d})$ are polynomials in $a_1, \ldots, a_r, a_2', \ldots, a_r'$ and $c_2$.
%\end{lem}
%\begin{proof}
%By the construction of $\Ub_{r,d}$ as a quotient of $\Omega_{r,d} \subset M_{r,d}$, the universal $\pp^1$-bundle $\P$ is equipped with an exact sequence of vector bundles
%\begin{equation} \label{eseq}
%0 \rightarrow (\pi^*\gamma^*T_d)(-1) \rightarrow \pi^*\gamma^*T_{r+d} \rightarrow \E \rightarrow 0.
%\end{equation}
%By the theorem on cohomology and base change, $R^1\pi_*(\pi^*T_d)(-1) = 0$, so pushing forward by $\pi$ induces an isomorphism 
%\[\gamma^*T_{r+d} \cong \pi_*\pi^* \gamma^* T_{r+d} \xrightarrow{\sim} \pi_* \E.\]
%On the other hand, tensoring with $\O_{\P}(-1)$ and pushing forward by $\pi$ induces an isomorphism
%\[\pi_*\E(-1) \xrightarrow{\sim} R^1\pi_*( (\pi^* \gamma^* T_d)(-2)) \cong \gamma^* T_d \otimes R^1 \pi_*  \O_{\P}(-2)  \cong \pi^* \gamma^* T_d.\]
%The middle isomorphism is the projection formula and the last isomorphism is Serre duality, noting that $\O_{\P}(-2) \cong \omega_{\P}$ because it is pulled back from the universal $\p^1$-bundle over $\BSL_2$, where this equality holds.
%Finally, 
%\end{proof}

\begin{example}[First Chern classes] \label{exc1}
Let $T_{\pi} = \O_{\P}(2)$ denote the relative tangent bundle of $\pi: \P \to \Ub_{r,d}$, so the the relative Todd class is $\mathrm{Td} _{\pi} = 1 + \frac{1}{2} c_1(T_{\pi}) + \ldots = 1 + z + \ldots$. Using Equation \eqref{idT}, and then Grothedieck--Riemann--Roch, we have that on $\Ub_{r,d}$,
\begin{align*} 
t_1 &= c_1(\pi_*\E(-1)) = \ch_1(\pi_*\E(-1)) = [\pi_*(\ch(\E) \cdot \ch(\O_{\P}(-1))\cdot \mathrm{Td}_{\pi})]_1 \\
&= [\pi_*(\ch(\E) \cdot (1 - z) \cdot (1 + z))]_1 = \pi_*(\ch_2(\E)) = \pi_*\left(\frac{1}{2}c_1(\E)^2 - c_2(\E)\right) \\
&= da_1 - a_2' \\
u_1 &= c_1(\pi_*\E) = \ch_1(\pi_*\E) = [\pi_*(\ch(\E) \cdot \mathrm{Td}_{\pi})]_1 = [\pi_*(\ch(\E) \cdot (1 + z))]_1 \\
&= \pi_*(\ch_2(\E) +  \ch_1(\E)z) = (da_1 - a_2') + a_1 \\
&= (d+1)a_1 - a_2'.
\end{align*}
It follows that $a_1 = u_1 - t_1$ and $a_2' =du_1 - (d+1)t_1$.
\end{example}

In Equation \eqref{asq}, we described $\Ub_{r,d} \times_{\BSL_2} \Ub_{s,e}$ as a quotient.
To similarly understand the moduli space of pairs of vector bundles on a $\p^1$-\emph{fibration}, we need the ``pair" version of $\Gamma_{r,d}$ and of Theorem \ref{BVthm}. Precisely, let us define
 $\Gamma_{r,d,s,e}$ to be the quotient of $G_{r,d,s,e} \times \GL_2$ by $t \mapsto (t\Id_d, \Id_{r+d}, t\Id_e, I_{s+e}, t^{-1}\Id_2)$.
 Then, we have
 \[\V_{r,d} \times_{\BPGL_2} \V_{s,e} = [\Omega_{r,d} \times \Omega_{s,e}/\Gamma_{r,d,s,e}].\]

Considering the commutative diagram
\begin{center}
\begin{tikzcd}
1 \arrow{r} & \mu_2 \arrow{d} \arrow{r} & G_{r,d,s,e} \times \SL_2 \arrow{d} \arrow{r} & \bullet  \arrow{d}{\sim} \arrow{r} &1 \\
1 \arrow{r} & \mathbb{G}_m \arrow{d}[swap]{(-)^2} \arrow{r} & G_{r,d,s,e} \times \GL_2 \arrow{d}{\det} \arrow{r} &\Gamma_{r,d,s,e} \arrow{r} &1 \\
&\mathbb{G}_m \arrow{r}[swap]{\mathrm{id}} &\mathbb{G}_m
\end{tikzcd}
\end{center}
we see by the snake lemma that $\Gamma_{r,d,s,e}$ is a $\mu_2$ quotient of $G_{r,d,s,e} \times \SL_2$.

Let us assume $r,s > 1$, so that 
the complement of $\Omega_{r,d} \times \Omega_{s,e} \subset M_{r,d} \times M_{s,e}$ has codimension at least $2$. In particular, by the excision and homotopy properties, we have natural identifications
\[\Pic(\Ub_{r,d} \times_{\BSL_2} \Ub_{s,e}) = \Pic (\mathrm{B}(G_{r,d,s,e} \times \SL_2)),\]
and
\[\Pic(\V_{r,d} \times_{\BPGL_2} \V_{s,e}) = \Pic(\mathrm{B}\Gamma_{r,d,s,e}).\]
The group $\Pic (\mathrm{B}(G_{r,d,s,e} \times \SL_2))$
is the free $\zz$-module generated by $t_1, u_1, v_1, w_1$ (see \eqref{gens}). Using Example \ref{exc1}, we see that the classes $a_1, a_2', b_1, b_2'$ also freely generate
$\Pic(\Ub_{r,d} \times_{\BSL_2} \Ub_{s,e})$.

\begin{lem} \label{int-pic}
The natural map $\Ub_{r,d} \times_{\BSL_2} \Ub_{s,e} \to \V_{r,d} \times_{\BPGL_2} \V_{s,e}$ induces an inclusion 
\[\Pic(\V_{r,d} \times_{\BPGL_2} \V_{s,e}) \hookrightarrow \Pic(\Ub_{r,d} \times_{\BSL_2} \Ub_{s,e}),\]
whose image is the subgroup generated by 
\begin{equation} \label{gens1}
\begin{cases} t_1, u_1, v_1, w_1 & \text{if $d, e$ both even} \\
2t_1, u_1, v_1, w_1 & \text{if $d$ odd and $e$ even} \\
t_1, u_1, 2v_1, w_1 & \text{if $d$ even and $e$ odd} \\
t_1 - v_1, 2t_1, u_1, w_1 & \text{if $d, e$ both odd,} \end{cases}
\end{equation}
or equivalently by
\begin{equation} \label{gens2}
\begin{cases}
a_1, a_2', b_1, b_2'  & \text{if $d, e$ both even}\\
2a_1, a_2', b_1, b_2'& \text{if $d$ odd and $e$ even}  \\
a_1, a_2', 2b_1, b_2' & \text{if $d$ even and $e$ odd}  \\
a_1 - b_1, 2a_1, a_2', b_2' & \text{if $d, e$ both odd.}
\end{cases}
\end{equation}
\end{lem}
\begin{proof}
Recall that $\Pic(\mathrm{B}G)$ is naturally identified with the character group of $G$ because it is identified with Mumford's functorial Picard group. %\hannah{right?}
The exact sequence of groups
\[0 \rightarrow \mu_2 \rightarrow G_{r,d,s,e} \times \SL_2 \rightarrow \Gamma_{r,d,s,e} \rightarrow 0\]
induces a left exact sequence
\begin{equation} \label{leftes}0 \rightarrow \Pic(\mathrm{B}\Gamma_{r,d,s,e}) \rightarrow \Pic(\mathrm{B}( G_{r,d,s,e} \times \SL_2)) \rightarrow \Pic(\mathrm{B}\mu_2).
\end{equation}
The Picard group  $\Pic(\mathrm{B}\mu_2)$ is isomorphic to $\zz/2\zz$. Let $h$ be a generator of $\Pic(\mathrm{B} \mu_2)$.
Recall that the map $\mu_2 \to G_{r,d,s,e} \times \SL_2$ sends $-1$ to $(-\Id_d, \Id_{r+d}, -\Id_e, \Id_{s+e}, -\Id_2)$. 
The generator $t_1 \in \Pic(\mathrm{B}(G_{r,d,s,e} \times \SL_2))$ corresponds to the determinant of the rank $d$ matrix. Thus, the right-hand map in \eqref{leftes} sends $t_1$ to $dh$. Similarly,
$u_1$ and $w_1$ are sent to zero, and $v_1$ to $eh$. The kernel is thus the subgroup generated by the classes listed in \eqref{gens1}.

The translation between \eqref{gens1} and \eqref{gens2} follows from Example \ref{exc1}. We explain the case $d, e$ both odd, the other cases being similar but simpler. Since $d$ and $e$ are both odd, the following change of basis matrix has integer coefficients
\[\left(\begin{matrix} t_1 - v_1 \\ 2t_1 \\ u_1 \\ w_1 \end{matrix}\right) = \left(\begin{matrix} e & \frac{d - e}{2} & -1 & 1 \\[3pt]
0 & d & -2 & 0 \\[3pt]
0 & \frac{d+1}{2} & -1 & 0 \\[3pt]
-(e+1) & \frac{e+1}{2} & 0 & -1
\end{matrix} \right)\left(\begin{matrix} a_1 - b_1 \\ 2a_1 \\ a_2' \\ b_2' \end{matrix}\right). \]
The determinant of the $4 \times 4$ matrix above is $1$ so the entries of the two column vectors generate the same subgroup with $\zz$-coefficients.
\end{proof}

\begin{lem} \label{comparison}
Let $X' \subset X$ be an open substack. Given a smooth map $f: Y \to X$, let $Y' \subset Y$ be the preimage of $X'$. 
If $\Pic(X) \to \Pic(Y)$ is injective, then $\Pic(X') \to \Pic(Y')$ is injective.
\end{lem}
\begin{proof}
It suffices to treat the 
case when $X' \subset X$ is the complement of an irreducible divisor $D$.  (Removing any component of codimension two or more from $X$ does not change $\Pic(X)$; if $X'$ is the complement of a reducible divisor, then we just apply the irreducible case to each component in turn.)

Because $f: Y \to X$ is smooth, $Y' \subset Y$ is the complement of the irreducible divisor $f^{-1}(D)$, which has class $f^*[D]$.
Let $\langle [D] \rangle$ denote the subgroup of $\Pic(X)$ generated by the fundamental class of $D$ and similarly for $\langle f^*[D] \rangle$ inside $\Pic(Y)$.
We therefore have a diagram of exact sequences where the left vertical map is surjective.
\begin{center}
\begin{tikzcd}
0 \arrow{r} &\langle[D] \rangle \arrow{r} \arrow{d} & \Pic(X) \arrow[hook]{d} \arrow{r} & \Pic(X') \arrow{r} \arrow{d} & 0 \\
0 \arrow{r} &\langle f^*[D] \rangle \arrow{r} & \Pic(Y) \arrow{r} & \Pic(Y') \arrow{r} & 0.
\end{tikzcd}
\end{center}
The result now follows from the snake lemma.
\end{proof}

We shall be applying Lemma \ref{comparison} in the context of a smooth map $Y \to X$ which is induced by a base change $\BSL_2 \to \BPGL_2$.
The basic idea is that injectivity of a certain map of Picard groups will allow us to argue that our previous calculations determining relations in $\Pic(\H_{k,g})$ actually hold in $\Pic(\Hp_{k,g})$  with the ``same formulas." We then just need to understand which classes are integral, which will be deduced from Lemma \ref{int-pic}.

\subsection{Construction of the base stacks}
To keep track of the integrality conditions arising from Lemma \ref{int-pic}, we shall find it useful to use the quantity
\begin{equation}\label{epsilondef}
\epsilon := \epsilon_{k,g} = \begin{cases} 1 & \text{if $g+k-1$ is even} \\ 2 & \text{if $g+k-1$ is odd,} \end{cases} 
\end{equation}
which keeps track of the parity of the degree of certain vector bundles on $\pp^1$ we associate to degree $k$, genus $g$ covers.
We shall often drop the subscript when $k$ and $g$ are understood. Below, we introduce the ``base stacks" $\mathscr{B}_{k,g}$ that parametrize the bundles on $\pp^1$ we shall associate to degree $k$, genus $g$ covers in Sections \ref{trigonal}, \ref{tetragonal}, and \ref{pentagonal}.

\subsubsection{Degree 3}
We set $\mathscr{B}_{3,g} := \V_{2,g+2}$. Then, by Lemma \ref{int-pic}, we have \begin{equation} \label{pb3} \Pic(\mathscr{B}_{3,g}) = \zz (\epsilon a_1) \oplus \zz a_2',
\end{equation}
where $\epsilon$ is as in \eqref{epsilondef}.

\subsubsection{Degree 4} \label{b4def}
First let us recall a standard, but very useful fact.
%\sam{not sure where to put this}
\begin{lem} \label{gm}
Suppose that $X\rightarrow Y$ is a $\gg_m$-torsor with associated line bundle $\mathcal{L}$. Then,
\[
\Pic(X)\cong \Pic(Y)/\langle c_1(\mathcal{L})\rangle.
\]
\end{lem}
\begin{proof}
Under the correspondence between $\gg_m$-torsors and line bundles, $X$ is the complement of the zero section in the total space of the line bundle $\mathcal{L}\rightarrow Y$. We thus have the exact sequence
\[
\Pic(Y)\rightarrow \Pic(\mathcal{L})\rightarrow \Pic(X)\rightarrow 0.
\]
After identifying $\Pic(\mathcal{L})$ with $\Pic(Y)$, the first map in the sequence is given by multiplication by the first Chern class $c_1(\mathcal{L})$, and the result follows.
\end{proof}

Let $\pi: \mathscr{P} \to \V_{3,g+3} \times_{\BPGL_2} \V_{2,g+3}$ be the universal $\pp^1$-fibration, equipped with universal bundles $\mathscr{E}$ of rank $3$ and $\mathscr{F}$ of rank $2$. Define $\mathscr{B}_{4,g}$ be the $\gg_m$-torsor over $\V_{3,g+3} \times_{\BPGL_2} \V_{2,g+3}$ associated to $\pi_*(\det \mathscr{E} \otimes \det \mathscr{F}^\vee)$. This push forward is a line bundle by cohomology and base change, and has class $c_1(\pi_*(\det \mathscr{E} \otimes \det \mathscr{F}^\vee)) = a_1 - b_1 \in \Pic(\V_{3,g+3} \times_{\BPGL_2} \V_{2,g+3})$.
The stack $\mathscr{B}_{4,g}$ parametrizes pairs of vector bundles $(E, F)$ on $\pp^1$-fibrations together with an isomorphism of their determinants.

Combining Lemma \ref{gm} and Lemma \ref{int-pic} (where we are necessarily either in the first or last case of \eqref{gens2}), we find
\begin{equation} \label{pb4} \Pic(\mathscr{B}_{4,g}) = \zz (\epsilon a_1) \oplus  \zz a_2' \oplus \zz b_2',
\end{equation}
where again, $\epsilon$ is as in \eqref{epsilondef}.

\subsubsection{Degree 5} \label{b5d}
Let $\pi:\mathscr{P}\rightarrow \V_{4,g+4}\times_{\BPGL_2}\V_{5,2g+8}$ be the universal $\p^1$-fibration, equipped with universal bundles $\mathscr{E}$ of rank $4$, degree $g+4$, and $\mathscr{F}$ of rank $5$ degree $2g+8$. Define $\mathscr{B}_{5,g}$ to be the $\gg_m$-torsor over $\V_{4,g+4}\times_{\BPGL_2}\V_{5,2g+8}$ associated to the bundle $\pi_*(\det \mathscr{E}^{\otimes2}\otimes \det \mathscr{F}^{\vee})$, which is a line bundle by cohomology and base change. It has first Chern class $2a_1-b_1$. The stack $\mathscr{B}_{5,g}$ parametrizes pairs of vector bundles $(E,F)$ on a $\p^1$-fibration together with an isomorphism between $(\det E)^{\otimes 2}$ and $\det F$. By Lemmas \ref{int-pic} and \ref{gm},
\begin{equation} \label{bp5}
\Pic(\mathscr{B}_{5,g})=\zz(\epsilon a_1)\oplus \zz a_2'\oplus \zz b_2'.
\end{equation}

\section{Trigonal}\label{trigonal}
The Picard group of $\Hp_{3,g}$ was computed by Bolognesi--Vistoli \cite[Theorem 1.1]{BV} when $g\geq 3$. As a warm-up for $k = 4,5$, we give a slightly different proof of their result; we also treat the case $g = 2$, and compute $\Pic(\Hp_{3,2})$.

Let us recall the linear algebraic data associated to a degree $3$ cover, as developed by Miranda \cite{M}, and later Casnati--Ekedahl \cite{CE}. For more details in our context see also \cite{BV} and \cite[Section 3.1]{part1}.
Given a degree $3$ cover $\alpha: C \to \pp^1$, we associate a rank $2$ vector bundle $E_\alpha := (\alpha_*\O_C/\O_{\pp^1})^\vee$ on $\pp^1$.
The cover naturally factors through an embedding $C \subset \pp E_\alpha^\vee \to \pp^1$ and $C \subset \pp E_\alpha^\vee$ is defined as the zero locus of a section
\[\eta_\alpha \in H^0(\pp^1, \det E_\alpha^\vee \otimes \Sym^3 E_\alpha) \cong H^0(\pp E_{\alpha}, \det E_\alpha^\vee \otimes \O_{\pp E_\alpha}(3)).\]
The association of $\alpha$ with $E_\alpha$ defines a map $\Hp_{3,g} \to \mathscr{B}_{3,g} := \mathscr{V}_{2,g+2}$. Let $\pi: \mathscr{P} \to \mathscr{B}_{3,g}$ be the universal $\pp^1$-fibration, equipped with universal rank $2$ bundle $\mathscr{E}$. We define $\mathscr{B}_{3,g}'$ to be the locus where $\det \mathscr{E}^\vee \otimes \Sym^3 \mathscr{E}$ is globally generated on the fibers of $\pi$. Equivalently, $\mathscr{B}_{3,g}'$ is the locus where
$\mathscr{E}$ has splitting type $(e_1, e_2)$ for $e_1 \leq e_2$ and $2e_1 - e_2 \geq 0$ on fibers of $\pi$.

Let $\mathscr{X}_{3,g}'$ be the total space of the vector bundle $\pi_*(\det \mathscr{E}^\vee \otimes \Sym^3 \mathscr{E})|_{\mathscr{B}_{3,g}'}$, which is locally free by the theorem on cohomology and base change, on $\mathscr{B}_{3,g}'$. 
Arguing exactly as in \cite[Lemma 5.1]{part1}, one sees that the association of $\alpha: C \to \pp^1$ with $(E_\alpha, \eta_\alpha)$ defines an open embedding of $\mathscr{H}_{3,g}$ into $\mathscr{X}_{3,g}'$. 
Let $\mathscr{D}_{3,g} := \mathscr{X}_{3,g}' \smallsetminus \mathscr{H}_{3,g}$ be the closed complement.

At this point, we have described stacks and morphisms
\[\mathscr{D}_{3,g} \to \mathscr{X}_{3,g}' \rightarrow \mathscr{B}_{3,g}' \to \mathscr{B}_{3,g} = \mathscr{V}_{2,g+3} \to \BPGL_2,\]
Base changing by $\BSL_2 \to \BPGL_2$, we obtain the stacks and morphisms studied in \cite[Section 4.2]{part2}.
\[\Delta_{3,g} \to \mathcal{X}_{3,g}' \rightarrow \mathcal{B}_{3,g}' \to \mathcal{B}_{3,g} = \mathcal{V}_{2,g+2} \to \BSL_2.\]
In \cite[Equation 4.5]{part2}, we showed that $\Delta_{3,g}$ is irreducible with fundamental class given by
\[[\Delta_{3,g}] =(8g + 12)a_1 - 9a_2' \in \Pic(\mathcal{X}_{3,g}') \cong \Pic(\mathcal{B}_{3,g}') .\]

Recall that $\Pic(\V_{2,g+2}) \to \Pic(\mathcal{V}_{2,g+2})$ is injective (Lemma \ref{int-pic}).
Applying Lemma \ref{comparison}, we also see that $\Pic(\mathscr{B}'_{3,g}) \to \Pic(\mathcal{B}'_{3,g})$ is injective, and then so too is $\Pic(\mathscr{X}'_{3,g}) \to \Pic(\mathcal{X}'_{3,g})$. Since $\mathscr{D}_{3,g}$ pulls back to $\Delta_{3,g}$,
it follows that $\mathscr{D}_{3,g}$ is irreducible of class
\[[\mathscr{D}_{3,g}] = \frac{8g+12}{\epsilon}(\epsilon a_1) - 9a_2' \in \Pic(\mathscr{X}_{3,g}') \cong \Pic(\mathscr{B}_{3,g}').\]
where $\epsilon$ is as in \eqref{epsilondef}. (Recall $\epsilon$ is always $1$ or $2$ so the coefficient above is an integer.)

Now, by excision, we have
\[\Pic(\mathscr{H}_{3,g}) = \frac{\Pic(\mathscr{X}_{3,g}')}{\langle [\mathscr{D}_{3,g}] \rangle} =  \frac{\Pic(\mathscr{B}_{3,g}')}{\langle [\mathscr{D}_{3,g}]\rangle}.\]
When $g > 2$, the complement of $\mathscr{B}_{3,g}' \subset \mathscr{B}_{3,g}$ has codimension at least $2$ by \cite[p. 16]{part2}. Hence, $\Pic(\mathscr{B}_{3,g}') = \Pic(\mathscr{B}_{3,g}) = \zz(\epsilon a_1) \oplus \zz a_2'$. Therefore, for $g > 2$, we have
\[\Pic(\mathscr{H}_{3,g}) = \frac{\zz(\epsilon a_1) \oplus \zz a_2'}{\langle(\frac{8g+12}{\epsilon})a_1 - 9a_2' \rangle }
    \cong \begin{cases} 
         \zz & \text{if $g \neq 0 \pmod 3$ and $g\neq 2$} \\ \zz \oplus \zz/3 \zz & \text{if $g = 0 \pmod 3$ and $g \neq 3 \pmod 9$} \\ \zz \oplus \zz/9 \zz & \text{if $g = 3 \pmod 9$.} \end{cases}\]

Meahwhile, when $g = 2$, the complement of $\mathscr{B}_{3,g}' \subset \mathscr{B}_{3,g}$ is an irreducible divisor corresponding to the locus where the universal bundle $\mathscr{E}$ has splitting type $(1, 3)$ on the fibers of $\mathscr{P} \to \mathscr{B}_{3,g}$.
As found in the proof of \cite[Lemma 4.3]{part1}, the (pullback to $\Pic(\B_{3,g})$ of the) class of this splitting locus is
\begin{equation} \label{s13}
s_{1,3} = a_2' - 2a_1 \in \Pic(\mathscr{B}_{3,g}) \subseteq \Pic(\mathcal{B}_{3,g}).
\end{equation}
Hence, noting that $\epsilon = 1$ when $g = 2$, we have
\[\Pic(\Hp_{3,2}) = \frac{\zz a_1 \oplus \zz a_2'}{\langle 28 a_1 - 9a_2', a_2' - 2a_1\rangle} \cong \zz/10\zz.\]

\subsection{Generating line bundles}\label{gb3}
One natural class on $\Hp_{3,g}$ is $\lambda :=c_1(f_* \omega_f)$ (which is pulled back from $\M_g$). Using 
 Example \ref{exc1}, we compute that the pullback of $\lambda$ to $\H_{3,2}$ is
\[\lambda = c_1(f_*\omega_f) = c_1(\pi_*(\alpha_* \omega_\alpha) \otimes \omega_\pi) = c_1(\pi_*\E(-2)) = (g+1)a_1 - a_2'.\]
(See \eqref{univdiagram} for definitions of the maps $f, \alpha, \pi$).
Note that when $g$ is odd, the coefficient of $a_1$ is even, so using  Lemma \ref{int-pic}, this class lies in the subgroup $\Pic(\mathscr{H}_{3,g})$, as it must.
In the case of genus $2$, we have $\lambda = a_1 + s_{1,3}$ (where $s_{1,3}$ is the relation in \eqref{s13}), so $\lambda$ generates $\Pic(\Hp_{3,2})$.
This is not surprising: in \cite{V3}, Vistoli
computed the integral Chow ring of the stack $\M_2$ and found in particular that $\Pic(\M_{2}) = \zz/10 \zz$, generated by $\lambda$. This means that the pullback map $\Pic(\M_2) \to \Pic(\Hp_{3,2})$ is an isomorphism.
%\end{rem}

We note here a corollary of this fact for use in the later sections of the paper.
Let $\mathscr{P}ic^k$ be the universal Picard stack over $\M_2$. Over a scheme $S$, its objects are families of smooth curves $\mathcal{C}\rightarrow S$ of genus $2$ together with a line bundle $\mathcal{L}$ of relative degree $k$ on the fibers. The group $\gg_m$ injects into the automorphism group of every object by scaling the line bundle. One can form the so-called $\gg_m$-rigidifcation of $\mathscr{P}ic^k$, which is a stack $\mathscr{P}^{k}$ such that $\mathscr{P}ic^k\rightarrow \mathscr{P}^k$ is a $\gg_m$-banded gerbe. %For each $k \geq 3$, there is a natural map $\mathscr{H}_{k,2} \to \mathscr{P}^k$.

%, which factors through a Grassmann fibration over $\mathscr{P}^k$.

\begin{cor} \label{hack}
Let $\mathscr{P}^k \to \M_2$ be as above. Then the pullback map $\Pic(\M_2) \to \Pic(\mathscr{P}^k)$ is injective.
\end{cor}
\begin{proof}
There are natural isomorphisms $\mathscr{P}^k \cong \mathscr{P}^{k+2}$ (given by tensoring with the canonical), so it suffices to prove the claim for $k = 2$ and $k = 3$. When $k = 2$, the canonical line bundle gives a section $\M_2 \to \mathscr{P}^k$, so the pullback map must be injective. 

Now consider the case $k = 3$. Every degree $3$ line bundle on a genus $2$ curve has $2$ sections. Therefore, $\Hp_{3,2}$ is naturally an open substack inside $\mathscr{P}^3$. (It is the complement of the universal curve $\mathscr{C} \hookrightarrow \mathscr{P}^3$ embedded by summing each point with a canonical divisor.) The isomorphism $\Pic(\M_2) \to \Pic(\Hp_{3,2})$ factors through $\Pic(\M_2) \to \Pic(\mathscr{P}^3)$, so the latter must also be injective. 
\end{proof}

For $g \neq 2, 5$, however, $\lambda$ cannot be used as a generator of $\Pic(\Hp_{3,g})$. (This follows from the fact that
$\det \left(\begin{matrix} \frac{g+1}{\epsilon} & -1 \\ \frac{8g+12}{\epsilon} & -9 \end{matrix}\right) = \frac{g - 3}{\epsilon}$ is not a unit, unless $g = 2$ or $5$.)

To describe line bundles generating $\Pic(\Hp_{3,g})$,
let $\mathscr{E}$ be the universal rank $2$ bundle on $\pi: \mathscr{P} \to \mathscr{H}_{3,g}$ (recall $\mathscr{E} = (\alpha_* \O_{\mathscr{C}}/\O_{\mathscr{P}})^\vee$.)
For $g > 2$, we can generate the free part of $\Pic(\Hp_{3,g})$ by
\[\mathscr{L} = \begin{cases} \pi_*\left(\det \mathscr{E} \otimes \omega_{\pi}^{\otimes (g+2)/2}\right) & \text{if $g$ even} \\ \\
\pi_*\left((\det \mathscr{E})^{\otimes 2} \otimes \omega_{\pi}^{\otimes (g+2)} \right) & \text{if $g$ odd} \end{cases} \qquad \text{which has} \qquad c_1(\mathscr{L}_1) =  \epsilon a_1.\]
When $g = 0 \pmod 3$ and $g \neq 3 \pmod 9$, the torsion subgroup is generated by
\[\frac{8g+12}{3}a_1 - 3a_2'= 3\lambda - \left(\frac{g-3}{3\epsilon}\right) c_1(\mathscr{L}).\]
When $g = 3 \pmod 9$, the torsion subgroup is generated by
\[\frac{8g+12}{9} a_1 - a_2' = \lambda - \left(\frac{g - 3}{9\epsilon} \right) c_1(\mathscr{L}).\]

\subsection{Simple branching}
Let $T \subset \Hp_{3,g}$ be the divisor of covers with a point of triple ramification, as defined in the introduction (see Figure \ref{TD}). In \cite[Proposition 2.8]{DP2}, Deopurkar--Patel compute the class of $T$. In terms of our generators, we have
\[T = (24g+36)a_1 -24a_2'.\]

\begin{proof}[Proof of Corollary \ref{mc}(1)]
Using excision, for $g \geq 3$, we have
\[\Pic(\Hp_{3,g}^s) = \Pic(\Hp_{3,g})/\langle T \rangle = \frac{\zz(\epsilon a_1) \oplus \zz a_2'}{\langle \left(\frac{8g+12}{\epsilon}\right)\epsilon a_1 - 9a_2', \left(\frac{24g+36}{\epsilon}\right) \epsilon a_1 - 24a_2'\rangle}\]
Now observe that
\begin{equation} \label{cb}
\left(\begin{matrix} -8 & 3 \\ -3 & 1 \end{matrix}\right) \left(\begin{matrix} (8g+12)/\epsilon & -9 \\ (24g+36)/\epsilon & -24 \end{matrix}\right) = \left(\begin{matrix} (8g + 12)/\epsilon & 0 \\ 0 & 3 \end{matrix} \right).
\end{equation}
The matrix on the left of \eqref{cb} is invertible over $\zz$. Thus, $\Pic(\Hp_{3,g}^s)$ is the sum of two cyclic groups with orders given by the diagonal entries of the matrix on the right of \eqref{cb}.
This completes the proof for $g \geq 3$.

Finally, for $g = 2$, we already know $a_2'= 2a_1$ by \eqref{s13} and that $10a_1 = 0$ in $\Pic(\Hp_{3,g})$. When we remove $T$, this creates one additional relation $0 = T = 84a_1-24a_2' = 36a_1$ in $\Pic(\Hp_{3,g}^s)$. We have $\gcd(10,36) = 2$, so $\Pic(\Hp_{3,2}^s) = \zz/2\zz$.
\end{proof}

\section{Tetragonal}\label{tetragonal}

We begin by briefly recalling the linear algebraic data associated to a degree $4$ cover, as developed by Casnati--Ekedahl \cite{CE}. For more details in our context, see \cite[Section 3.2]{part1}. Given a degree $4$ cover $\alpha: C \to \pp^1$, we associate two vector bundles on $\pp^1$:
\[E_\alpha := (\alpha_*\O_C/\O_{\pp^1})^\vee = \ker(\alpha_*\omega_\alpha \to \O_{\pp^1}) \qquad \text{and} \qquad F_\alpha := \ker(\Sym^2 E_\alpha \to \alpha_*\omega_{\alpha}^{\otimes 2}).\]
The first is rank $3$ and the second is rank $2$.
If $C$ has genus $g$, then both bundles have degree $g+3$.
Geometrically, the curve $C$ is embedded in $\gamma:\p E^{\vee}_{\alpha}\rightarrow \p^1$ as the zero locus of a section
\[
\delta_{\alpha} \in H^0(\p E_{\alpha}^{\vee},\O_{\p E^{\vee}_{\alpha}}(2)\otimes \gamma^*F_{\alpha}^{\vee}).
\]
In each fiber of $\gamma$, the four points are the base locus of a pencil of conics parametrized by $F_\alpha$. We can also think of $\delta_{\alpha}$ as a section of a bundle on $\p^1$ through the natural isomorphism
\[
H^0(\p E_{\alpha}^{\vee},\O_{\p E^{\vee}_{\alpha}}(2)\otimes \gamma^*F_{\alpha}^{\vee})\cong H^0(\p^1, \Sym^2 E_{\alpha}\otimes F_{\alpha}^{\vee}).
\]
The cover $\alpha$ also determines an isomorphism $\phi_\alpha: \det E_\alpha \cong \det F_\alpha$ (see \cite[Section 3.2]{part1})

The association of $\alpha:C\rightarrow \p^1$ with the triple $(E_{\alpha}, F_{\alpha}, \phi_\alpha)$ gives rise to a map of $\Hp_{4,g}$ to the base stack $\mathscr{B}_{4,g}$ defined in \ref{b4def}.
Unlike in the degree $3$ case, the map $\Hp_{4,g} \to \mathscr{B}_{4,g}$ does \emph{not} factor through a vector bundle over $\mathscr{B}_{4,g}$. Nevertheless, we shall define an open substack $\mathscr{H}_{4,g}'$ that does admit such a nice description. 
The key fact, to be established in Lemma \ref{k4}, is that the complement of $\Hp_{4,g}' \subset \Hp_{4,g}$ has codimension at least $2$ for all $g \neq 3$. Thus, it will suffice to compute $\Pic(\Hp_{4,g}')$.

\subsection{The open $\Hp'_{4,g}$}
First define $\mathscr{B}_{4,g}' := \mathscr{B}_{4,g}  \smallsetminus R^1\pi_*(\mathscr{F}^\vee \otimes \Sym^2 \mathscr{E})$.
We define $\Hp'_{4,g}\subset \Hp_{4,g}$ to be the base change of the map $\Hp_{4,g}\rightarrow \mathscr{B}_{4,g}$ along the open embedding $\mathscr{B}'_{4,g}\hookrightarrow \mathscr{B}_{4,g}$. 
The key property of $\Hp_{4,g}'$ is that the map $\Hp_{4,g}' \to \mathscr{B}_{4,g}'$ factors through an open inclusion in the total space of a vector bundle (by an argument identical to \cite[Lemma 5.3]{part1}):
\begin{equation} \label{v4}
\Hp'_{4,g}\hookrightarrow\mathscr{X}'_{4,g} := \pi_*(\mathscr{F}^{\vee}\otimes \Sym^2 \mathscr{E})|_{\mathscr{B}_{4,g}'}.
\end{equation}

As promised, we now show that the complement of $\Hp'_{4,g} \subset \Hp_{4,g}$ has codimension at least $2$ (except when $g = 3$).  This essentially follows from earlier work of Deopurkar--Patel.

\begin{lem} \label{k4}
If $g \neq 3$, every component of the complement of $\Hp'_{4,g}\subset \Hp_{4,g}$ has codimension at least $2$. Hence, there is an isomorphism $\Pic(\Hp_{4,g}) \cong \Pic(\Hp_{4,g}')$.
\end{lem}
\begin{proof}
Following the notation of \cite{DP},
let $M(E, F) \subset \Hp_{4,g}$ denote the locus of covers $\alpha$ with $E_\alpha \cong E$ and $F_\alpha \cong F$. 
The complement of $\Hp'_{4,g}\subset \Hp_{4,g}$ is the union of $M(E, F)$ such that
\begin{equation} \label{cc4}
h^1(\pp^1, F^\vee \otimes \Sym^2 E) > 0.
\end{equation}
If $E = \O(e_1) \oplus \O(e_2) \oplus \O(e_3)$ with $e_1 \leq e_2 \leq e_3$ and $F = \O(f_1) \oplus \O(f_2)$ with $f_1 \leq f_2$, then \eqref{cc4} is equivalent to $2e_1 - f_1 \leq -2$.

Next, let
$E_{\gen}$ and $F_{\gen}$ denote the balanced bundles of rank $3$ and $2$ and degree $g+3$.
First note that $h^1(F_{\gen}^\vee \otimes \Sym^2 E_{\gen}) = 0$: this is equivalent to $2 \lfloor \frac{g+3}{3} \rfloor - \lceil \frac{g+3}{2} \rceil \geq -1$.
This says that $M(E_{\gen}, F_{\gen}) \subseteq \Hp_{4,g}'$. Hence, any divisorial component of $\Hp_{4,g} \smallsetminus \Hp_{4,g}'$ is contained in a divisorial component of $\Hp_{4,g} \smallsetminus M(E_{\gen}, F_{\gen})$.

Next, let us define bundles that are ``one-off" from balanced
\begin{align*}
    F_1 &:= \O(n-1) \oplus \O(n+1) & \qquad & \text{if $n = \frac{g+3}{2}$ is an integer} \\
    E_1 &:= \O(m-1) \oplus \O(m) \oplus \O(m+1) & \qquad &\text{if $m = \frac{g+3}{3}$ is an integer.}
\end{align*}
In \cite[p. 20]{DP}, Deopurkar--Patel enumerate the divisorial components of $\Hp_{4,g} \smallsetminus M(E_{\gen}, F_{\gen})$ and show that, for $g \neq 3$, they are always of the form \begin{align}
    &\overline{M(E_{\gen}, F_1)} &\qquad &\text{if $2 \mid g+3$} \label{oo} \\
    &\overline{M(E_1, F_{\gen})} &\qquad &\text{if $3 \mid g+3$}. \label{ooo}
\end{align}
Note that we are using the irreducibility of $M$ and $CE$ in \cite[Propositions 4.5 and 4.7]{DP} to write these divisors as the closures above.
Meanwhile, when $g = 3$, the stratum
\begin{equation} \label{weird3}
\overline{M(E_1, F_1)} = \overline{M(\O(1) \oplus \O(2) \oplus \O(3), \O(2) \oplus \O(4))} 
\end{equation}
is also a divisor. This divisor does not lie in $\Hp_{4,g}'$.

One readily checks that
\begin{equation} \label{van1}
h^1(F_1^\vee \otimes \Sym^2 E_{\gen}) = h^1(F_{\gen}^\vee \otimes \Sym^2 E_1) = 0.
\end{equation}
Hence, when they are defined, $\Hp_{4,g}'$ contains each of 
\[M(E_{\gen}, F_{\gen}), \quad M(E_{\gen}, F_1), \quad M(E_1, F_{\gen}),\]
and, for $g \neq 3$, all other possible $M(E, F)$ have codimension at least $2$.
\end{proof}

\begin{rem}
Each $M(E, F)$ can be constructed directly as a global quotient, giving rise to a bound on $\Pic(M(E, F))$.
Deopurkar--Patel use their enumeration of the components of $\H_{4,g} \smallsetminus M(E_{\gen}, F_{\gen})$ to count ranks and prove the Picard rank conjecture for $k \leq 5$. 

The new innovation in our work is that we have built a larger open $\Hp_{4,g}'$ which contains several $M(E, F)$
and, in particular, is not missing any divisorial components (when $g \neq 3$). Hence, we see $\Pic(\Hp_{4,g}) = \Pic(\Hp_{4,g}')$, and we can compute the later \emph{integrally} using excision on $\Hp_{4,g}' \subset \mathscr{X}_{4,g}'$.
\end{rem}

\begin{lem} \label{g3lem}
When $g = 3$, the complement of $\Hp_{4,3}' \subset \Hp_{4,3}$ is an irreducible divisor whose class lies in the subgroup generated by $\epsilon a_1$ and $a_2'$ (these classes are defined on all of $\Hp_{4,3}$ via pull back along $\Pic(\mathscr{B}_{4,3}) \to \Pic(\mathscr{H}_{4,g})$.) 
\end{lem}
\begin{proof}
Continuing the notation of the previous lemma, first note that if \[M(\O(1) \oplus \O(2) \oplus \O(3), \O(f_1) \oplus \O(f_2)) \neq \varnothing\]
then by \cite[Proposition 5.6(2)]{part1}, we have $f_1 \leq 2$ and $f_2 \leq 4$ and $f_1 + f_2 = 6$, hence $f_1 = 2$ and $f_2 = 4$.
The divisor in \eqref{weird3} can therefore be viewed as the locus were the universal $\mathscr{E}$ (pulled back along $\Hp_{4,3} \to \mathscr{B}_{4,3}$) has splitting type $(1, 2, 3)$ on fibers of the universal $\pp^1$-fibration. 
 As a splitting locus for $\mathscr{E}$, this divisor occurs in the expected codimension, so its fundamental class is determined by the universal splitting loci formulas of \cite{L}. In particular, it can be expressed in terms of the classes $\epsilon a_1, a_2'$.
\end{proof}

\subsection{Excision}
Recall the inclusion of \eqref{v4} and let $\mathscr{D}_{4,g} := \mathscr{X}'_{4,g} \smallsetminus \Hp'_{4,g}$ be the complement.
By excision, we have a series of surjections (the middle map is an isomorphism because $\mathscr{X}'_{4,g}$ is a vector bundle over $\mathscr{B}_{4,g}'$):
\begin{equation} \label{s4}
\Pic(\mathscr{B}_{4,g}) \rightarrow  \Pic(\mathscr{B}_{4,g}') \cong \Pic(\mathscr{X}_{4,g}') \rightarrow \Pic(\mathscr{H}_{4,g}').
\end{equation}
Moreover, the fundamental class $[\mathscr{D}_{4,g}] \in \Pic(\mathscr{X}'_{4,g})$ lies in the kernel of the last map in \eqref{s4} (and it generates the kernel when $\mathscr{D}_{4,g}$ is irreducible.)

At this point we have defined a sequence of morphisms
\begin{equation} \label{scr4}\mathscr{D}_{4,g} \to \mathscr{X}_{4,g}' \rightarrow \mathscr{B}_{4,g}' \rightarrow \mathscr{B}_{4,g} \rightarrow \mathscr{V}_{3,g+3} \times_{\BPGL_2} \mathscr{V}_{2,g+3} \rightarrow \BPGL_2. 
\end{equation}

\begin{lem} \label{allc}
For $g \geq 2$, some combination of components of $\mathscr{D}_{4,g}$ has class
\[\frac{8g+20}{\epsilon}(\epsilon a_1) - 8a_2' - b_2' \in \Pic(\mathscr{X}'_{4,g}) \cong \Pic(\mathscr{B}_{4,g}'). \]
In particular, an integral relation holds in $\Pic(\Hp_{4,g}')$ expressing $b_2'$ in terms of $\epsilon a_1$ and $a_2'$.
\end{lem}
\begin{proof}
Base changing \eqref{scr4} by $\BSL_2 \to \BPGL_2$, we obtain the stacks and morphisms considered in \cite[Section 5.2]{part1} (below $\Delta_{4,g}'$ is the complement of the open inclusion $\H_{4,g}' \hookrightarrow \X_{4,g}'$ of \cite[Lemma 5.3]{part1}):
\[\Delta_{4,g}' \to \mathcal{X}_{4,g}' \rightarrow \mathcal{B}_{4,g}' \rightarrow \mathcal{B}_{4,g} \rightarrow \mathcal{V}_{3,g+3} \times_{\BSL_2} \mathcal{V}_{2,g+3} \rightarrow \BSL_2. \]
We claim some combination of components of $\Delta_{4,g}'$ has class
\[(8g + 20)a_1 - 8a_2' - b_2'\in \Pic(\mathcal{X}'_{4,g}) \cong \Pic(\mathcal{B}_{4,g}). \]
 This will establish the lemma since the map $\Pic(\mathscr{X}'_{4,g}) \to \Pic(\mathcal{X}'_{4,g})$
 sends the class of a component of $\mathscr{D}_{4,g}$ to the class of the corresponding component of $\Delta_{4,g}'$. Because $\Pic(\mathscr{X}'_{4,g}) \to \Pic(\mathcal{X}'_{4,g})$  is injective (by Lemma \ref{comparison}), classes are represented by the same formulas in either group.

By \cite[Lemma 5.2 and Equation 5.7]{part2}, we know that $(8g + 20)a_1 - 8a_2' - b_2' = 0$ in $\Pic(\H_{4,g})$, so this relation must also hold on the open substack $\H_{4,g}' \subset \H_{4,g}$. But, we also know $\H_{4,g}' \cong \mathcal{X}'_{4,g} \smallsetminus \Delta_{4,g}'$. By excision, every relation among $a_1, a_2', b_2'$ restricted to $\Pic(\H_{4,g}')$ comes from a class supported on $\Delta_{4,g}' \subset \mathcal{X}'_{4,g}$. That is, some combination of components of $\Delta_{4,g}'$ has class $(8g + 20)a_1 - 8a_2' - b_2' = 0$. The corresponding
combination of components of $\mathscr{D}_{4,g}$ will have the same class.
\end{proof}
\begin{rem}
In fact, the fundamental class of $\mathscr{D}_{4,g}$ 
has the class displayed in Lemma \ref{allc}.
For a more conceptual explanation, we sketch the following argument.
Recall that in \cite[Equation 5.7]{part2}, we computed the restriction of $[\Delta_{4,g}']$ to a slightly smaller open $\mathcal{X}_{4,g}^\circ \subset \mathcal{X}'_{4,g}$ via principal parts bundle techniques. For $g$ sufficiently large, the complement of $\mathcal{X}_{4,g}^\circ \subset \mathcal{X}'_{4,g}$ has codimension at least $2$, so the codimesnion $1$ calculation holds on all of $\mathcal{X}'_{4,g}$.
That is $[\Delta_{4,g}'] = (8g+20)a_1 - 8a_2' - b_2'$ and so $[\mathscr{D}_{4,g}]$ also has this class.

But even for smaller $g$, 
we claim the formula for $[\Delta_{4,g}]$ in \cite[Equation 5.7]{part2} holds on all of $\mathcal{X}_{4,g}'$. Although the principal parts map of \cite[Equation 5.4]{part2} need not be surjective over all of $\mathcal{X}_{4,g}'$ (so the vanishing locus $\tilde{\Delta}_{4,g}$ of the principal parts bundle map need not be a vector bundle) 
the calculation of the fundamental class holds so long as $\tilde{\Delta}_{4,g}$ has the correct codimension. One can verify this by stratifying the base by loci where the rank drops and checking that the strata where the rank drops by $\delta$ have codimension greater than $\delta$. This also establishes irreducibility of $\Delta_{4,g}'$ and $\mathscr{D}_{4,g}$ when $g \geq 4$. However, this fact is not actually necessary for our argument.
\end{rem}

\begin{proof}[Proof of Theorem \ref{Pic}(1) for $g \geq 4$]
By excision, we know that
\[\Pic(\Hp_{4,g}') \cong \frac{\Pic(\mathscr{X}_{4,g}')}{\text{classes supported on $\mathscr{D}_{4,g}$}}.\]
Meanwhile, $\Pic(\mathscr{X}_{4,g}')$ is a quotient of $\Pic(\mathscr{B}_{4,g}) =\zz (\epsilon a_1) \oplus \zz a_2' \oplus \zz b_2'$ (see \eqref{pb4}).
Hence, using Lemma \ref{allc}, we see $\Pic(\Hp_{4,g}')$ is a quotient of
\[\frac{\zz (\epsilon a_1) \oplus \zz a_2' \oplus \zz b_2'}{\langle \frac{8g+20}{\epsilon}(\epsilon a_1) - 8a_2' - b_2' \rangle} \cong \zz (\epsilon a_1) \oplus \zz a_2'. \]

Recall that in Lemma \ref{k4}, we showed $\Pic(\Hp_{4,g}) = \Pic(\Hp_{4,g}')$ for $g \geq 4$. By \cite[Proposition 2.15]{DP}, this group has rank at least $2$. Therefore, $\Pic(\Hp_{4,g}) = \zz \oplus \zz$, since any quotient would have smaller rank.
\end{proof}

\subsubsection{Genus $3$}
When $g=3$, we require a different argument, as Lemma \ref{g3lem} tells us that $\Hp_{4,3}' \subset \Hp_{4,3}$ is the complement of a divisor.  

\begin{proof}[Proof of Theorem \ref{Pic}(1) when $g=3$]
By Lemma \ref{g3lem}, the  kernel of the restriction map \[
\Pic(\Hp_{4,3}) \to \Pic(\Hp_{4,3}')\]
lies in the subgroup $\langle a_1, a_2' \rangle$. We know that $\Pic(\Hp'_{4,g})$ is generated by the classes $a_1, a_2', b_2'$, so it follows that $\Pic(\Hp_{4,3})$ is also generated by these $3$ classes.

Next, we claim that $b_2'$ is integrally expressible in terms of $a_1, a_2'$.
By Lemma \ref{allc}, we know $b_2'$ is expressible in terms of $a_1, a_2'$ in $\Pic(\Hp_{4,3}')$.
But, the kernel of $\Pic(\Hp_{4,3}) \to \Pic(\Hp_{4,3}')$ lies in $\langle a_1, a_2' \rangle$, so $b_2'$ must also be expressible in terms of $a_1, a_2'$ in $\Pic(\Hp_{4,3})$.

It follows that $\Pic(\Hp_{4,3})$ is a quotient of $\zz a_1 \oplus \zz a_2'$.
However, by \cite{DP}, we know that $\Pic(\Hp_{4,3})$ has rank $2$. Any further quotient would have lower rank, so we are done.
\end{proof}

\subsubsection{Genus $2$}\label{4genus2}
The proof of \cite[Proposition 2.15]{DP} (showing $\Pic(\Hp_{4,g})$ has rank $2$) does not go through when $g = 2$ because Deopurkar--Patel's test family $B_3$ has curves with disconnecting nodes, so it does not lie in their $\widetilde{\H}_{4,2}^{ns}$. However, their proof does establish that the rank of $\Pic(\Hp_{4,2})$ is at least $1$. This, together with Lemma \ref{hack}, provides enough information to determine the Picard group.

\begin{proof}[Proof of Theorem \ref{Pic}(1) when $g=2$]
We have already established that $\Pic(\Hp_{4,2})$ is generated by $2a_1$ and $a_2'$.  Using Example \ref{exc1}, we compute that (see \eqref{univdiagram} for definitions of the maps $f, \alpha, \pi$)
\[\lambda = c_1(f_*\omega_f) = c_1(\pi_*(\alpha_* \omega_\alpha) \otimes \omega_\pi) = c_1(\pi_*\E(-2)) = 4a_1 - a_2' = 2(2 a_1) - a_2'.\]
From this we see that $\lambda$ and $2a_1$ are generators for $\Pic(\Hp_{4,2})$. Since $\lambda$ is the generator of $\Pic(\M_2) \cong \zz/10\zz$, we see that $\Pic(\Hp_{4,2})$ is a quotient of $\zz \oplus \zz/10$. 

By the discussion at the start of this section, we know $\Pic(\Hp_{4,2})$ has rank at least $1$.
Thus, it remains to prove that $\Pic(\M_2) \to \Pic(\Hp_{4,2})$ is injective.

Let $\mathscr{P}^k$ be the universal Picard variety over $\M_2$ as in Section \ref{gb3} (the $\gg_m$-rigidification of the universal Picard stack $\mathscr{P}ic^k$).
As in \cite[Section 6]{Mochizuki}, the natural map $\Hp_{4,2} \to \mathscr{P}^4$ factors through a Grassmann fibration. For this, recall that every degree $4$ line bundle on a genus $2$ curve has a $3$-dimensional space of sections.
Let $\mathscr{G} \to \mathscr{P}^4$ be the Grassmann fibration 
parametrizing two-dimensional subspaces of the space of global sections of a degree $4$ line bundle. Then 
$\mathscr{H}_{4,2}$ sits naturally as an open substack $\mathscr{G}$.
%let $\mathscr{L}$ be the universal line bundle on $\mathscr{P}ic^4 \times_{\M_2} \mathscr{C}$, and let $\nu: \mathscr{P}ic^4 \times_{\M_2} \mathscr{C} \to \mathscr{P}ic^4$ be the projection. 
%Every degree $4$ line bundle on a genus $2$ curve has a $3$-dimensional space of sections, so $\nu_* \mathscr{L}$ is a rank $3$ vector bundle on $\mathscr{P}ic^4$. The Hurwitz space $\Hp_{4,2}$ sits naturally as an open inside $G(2, \nu_* \mathscr{L})$. 
Its complement $Z = \mathscr{G} \smallsetminus \Hp_{4,2}$ is the locus of pencils with a base point. Note that $Z$ has $1$-dimensional irreducible fibers over $\mathscr{P}^4$, so $Z$ is irreducible. Since $Z$ meets each fiber of $\mathscr{G} \to  \mathscr{P}^4$, it is not equivalent to the pullback of a divisor on $\mathscr{P}^4$. In particular, the map $\Pic(\mathscr{P}^4) \to \Pic(\Hp_{4,2})$ must be injective. Using Lemma \ref{hack}, we conclude that $\Pic(\M_2) \to \Pic(\Hp_{4,2})$ is also injective, completing the proof.
\end{proof}

\begin{rem}
Geometrically, the fact that $\Pic(\mathscr{H}_{4,2})$ has rank $1$ can be explained by the fact that $\Delta_{4,2}'$ is reducible. Thus, its components give rise to further relations beyond just its fundamental class.
\end{rem}

\subsection{Generating line bundles}\label{gb4}
Let $\epsilon = 1$ if $g$ is odd and $\epsilon = 2$ if $g$ is even. We have shown that $\Pic(\Hp_{4,g})$ is generated by $\epsilon a_1$ and $a_2'$, or equivalently by $\epsilon a_1$ and $\lambda := (g+2) a_1 - a_2'$. Let $\pi: \mathscr{P} \to \Hp_{4,g}$ be the universal $\pp^1$-fibration and $\mathscr{E}$ the universal rank $3$, degree $g+3$ vector bundle on $\mathscr{P}$. Recall that $\omega_{\pi}$ has relative degree $-2$. Line bundles generating $\Pic(\Hp_{4,g})$ are given by
\[\mathscr{L}_1 = \begin{cases} \pi_*\left(\det \mathscr{E} \otimes \omega_{\pi}^{\otimes (g+3)/2}\right) & \text{if $g$ odd} \\ \\
\pi_*\left((\det \mathscr{E})^{\otimes 2} \otimes \omega_{\pi}^{\otimes (g+3)} \right) & \text{if $g$ even} \end{cases} \qquad \text{which has} \qquad c_1(\mathscr{L}_1) =  \epsilon a_1\]
and
\[\mathscr{L}_2 = \det f_*(\omega_f) = \det \pi_*(\mathscr{E} \otimes \omega_{\pi}) \qquad\text{which has} \qquad c_1(\mathscr{L}_2) = \lambda = (g+2) a_1 - a_2'.\]

\subsection{Simple branching}
Let $T$ and $D$ be the divisors in $\Hp_{4,g}$ as in the introduction (see Figure \ref{TD}).
In \cite[Lemma 7.6]{part2}, we wrote $T$ and $D$ (pulled back to $\Pic(\H_{4,g})$) in terms of our generators $a_1$ and $a_2'$:
\[T = (24g + 60)a_1 - 24a_2' \qquad D = -(32g+80)a_1 + 36a_2'.\]
Note that the coefficient of $a_1$ is a multiple of $\epsilon$, as it must be, since these classes are defined in $\Pic(\Hp_{4,g}) \subseteq \Pic(\H_{4,g})$.

\begin{proof}[Proof of Corollary \ref{mc}(2)]
By excision, we have $\Pic(\mathscr{H}_{4,g}^s) = \Pic(\Hp_{4,g})/\langle T, D \rangle$.
 Row operations over $\zz$ diagonalize the change of basis matrix for $\epsilon a_1, a_2'$ to $T, D$, namely we have
\[\left(\begin{matrix} 3 & 2 \\
4&3\end{matrix}\right)
\left(\begin{matrix} (24g+60)/\epsilon & -24 \\
(-32g-80)/\epsilon & 36\end{matrix}\right) = 
\left(\begin{matrix} (8g+20)/\epsilon & 0 \\
0 & 12\end{matrix}\right),
\]
where the matrix on the left is invertible over $\zz$.
By its definition in \eqref{epsilondef}, we have $\epsilon = 1$ when $g$ is odd and $\epsilon = 2$ when $g$ is even, so
the corollary follows for $g \geq 3$. 

Finally, when $g = 2$, the above tells us $18(2a_1) = 0$ and $12a_2' = 0$, and we have the additional relation 
\[0 = 10\lambda = 10(4a_1 - a_2') \qquad \Rightarrow \qquad 0 = 2(2a_1 + a_2').\]
Using generators $2a_1 + a_2'$ and $2a_1$, we see that they generate cyclic groups of order $2$ and $18$ respectively.
\end{proof}

\section{Pentagonal}\label{pentagonal}
We begin by recalling the linear algebraic data associated to degree $5$ covers, as developed by Casnati \cite{C}. For more details in our context, see \cite[Section 3.3]{part1}. To a degree $5$, cover $\alpha: C \to \pp^1$, we again associate two vector bundles on $\pp^1$:
\[E_\alpha := (\alpha_*\O_C/\O_{\pp^1})^\vee = \ker(\alpha_*\omega_\alpha \to \O_{\pp^1}) \qquad \text{and} \qquad F_\alpha := \ker(\Sym^2 E_\alpha \to \alpha_*\omega_{\alpha}^{\otimes 2}).\]
If $C$ has genus $g$, then $E_\alpha$ has degree $g+4$, and rank $4$, while $F_\alpha$ has degree $2g+8$ and rank $5$.
Geometrically, the curve $C$ is embedded in $\gamma:\p( E^{\vee}_{\alpha}\otimes \det E_{\alpha})\rightarrow \p^1$, which further maps to $\p(\wedge^2 F_{\alpha})$ via an associated section
\[
\eta_{\alpha}\in H^0(\p^1,\H om(E_{\alpha}^{\vee}\otimes \det E_{\alpha},\wedge^2 F_{\alpha})).
\]
The curve $C$ is obtained as the intersection of the image of $\p( E^{\vee}_{\alpha}\otimes \det E_{\alpha})$ with the Grassmann bundle $G(2,F_{\alpha}) \subset \pp(\wedge^2 F_\alpha)$. The cover $\alpha$ also determines an isomorphism $\phi_{\alpha}:\det E_{\alpha}^{\otimes 2}\rightarrow \det F_{\alpha}$, \cite[p. 10]{part1}.

The association of $\alpha:C\rightarrow \p^1$ with the triple $(E_{\alpha}, F_{\alpha}, \phi_\alpha)$ gives rise to a map $\Hp_{5,g}\rightarrow \mathscr{B}_{5,g}$, for the base stack $\mathscr{B}_{5,g}$ defined in \ref{b5d}. Just like in the degree $4$ case, the map $\Hp_{5,g}\rightarrow \mathscr{B}_{5,g}$ does not factor through a vector bundle over $\mathscr{B}_{5,g}$, but an open substack $\Hp_{5,g}'$ does. When $g\neq 3$, we will show that the complement of $\Hp_{5,g}'$ in $\Hp_{5,g}$ has codimension at least $2$, so it will suffice to compute $\Pic(\Hp_{5,g})$. We will then deal with the $g=3$ case separately.
\subsection{The open substack $\Hp'_{5,g}$}
First, define $\mathscr{B}'_{5,g}:=\mathscr{B}_{5,g}\smallsetminus R^1\pi_*(\H om(\mathscr{E}^\vee \otimes \det \mathscr{E}, \wedge^2 \mathscr{F}))$. Let $\Hp'_{5,g}$ be the base change of $\Hp_{5,g}\rightarrow \mathscr{B}_{5,g}$ along the open embedding $\mathscr{B}'_{5,g}\hookrightarrow \mathscr{B}_{5,g}$. Arguing exactly as in \cite[Lemma 5.3]{part1}, the morphism $\Hp'_{5,g}\rightarrow \mathscr{B}'_{5,g}$ factors through the total space of a vector bundle over $\mathscr{B}'_{5,g}$:
\[
\Hp'_{5,g}\hookrightarrow \mathscr{X}'_{5,g}:=\pi_*(\H om(\mathscr{E}^\vee \otimes \det \mathscr{E}, \wedge^2 \mathscr{F}))|_{\mathscr{B}'_{5,g}}.
\]

\begin{lem} \label{c5}
Suppose $g \neq 3$. Then, every component of the complement of $\Hp'_{5,g}\subset \Hp_{5,g}$ has codimension at least $2$. In particular, $\Pic(\Hp_{5,g}) = \Pic(\Hp_{5,g}')$
\end{lem}
\begin{proof}
Following the notation of \cite{DP},
let $M(E, F) \subset \Hp_{5,g}$ denote the locus of covers $\alpha$ with $E_\alpha \cong E$ and $F_\alpha \cong F$. 
The complement of $\Hp'_{5,g}\subset \Hp_{5,g}$ is the union of $M(E, F)$ such that
\begin{equation} \label{cc}
h^1(E \otimes \det E^\vee \otimes \wedge^2 F^\vee) > 0.
\end{equation}
If $E = \O(e_1) \oplus \cdots \oplus \O(e_4)$ with $e_1 \leq \cdots \leq e_4$ and $F = \O(f_1) \oplus \cdots \oplus \O(f_5)$ with $f_1 \leq \cdots \leq f_5$, then \eqref{cc} is equivalent to $e_1 + f_1 + f_2 - (g+4) \leq -2$.

Next, let
$E_{\gen}$, respectively $F_{\gen}$, denote the balanced bundle of rank $4$ and degree $g+4$, respectively rank $5$ and degree $2g+8$.
First note that $h^1(E_{\gen} \otimes \det E_{\gen}^\vee \otimes \wedge^2 F_{\gen}^\vee) = 0$, as in this case, $e_1 + f_1 + f_2 - (g+4) \geq \lfloor \frac{g+4}{4} \rfloor + 2 \lfloor \frac{2(g+4)}{5} \rfloor - (g+4) \geq -1$ (except when $g = 3$, in which case $f_2 = \lceil \frac{2(g+4)}{5}\rceil$, so we still have $e_1 + f_1+f_2 \geq -1$.)
This says that $M(E_{\gen}, F_{\gen}) \subseteq \Hp_{5,g}'$. Hence, any divisorial component of $\Hp_{5,g} \smallsetminus \Hp_{5,g}'$ is contained in a divisorial component of $\Hp_{5,g} \smallsetminus M(E_{\gen}, F_{\gen})$.

Again, let us define bundles that are ``one-off" from balanced
\begin{align*}
    F_1 &:= \O(n-1) \oplus \oplus \O(n)^{\oplus 3} \O(n+1) & \qquad & \text{if $n = \frac{2g+8}{5}$ is an integer} \\
    E_1 &:= \O(m-1) \oplus \O(m)^{\oplus 2} \oplus \O(m+1) & \qquad &\text{if $m = \frac{g+4}{4}$ is an integer.}
\end{align*}
In \cite[p. 25]{DP}, Deopurkar--Patel enumerate the divisorial components of $\Hp_{5,g} \smallsetminus M(E_{\gen}, F_{\gen})$ and show that, for $g \neq 3$, they are always of the form \begin{align*}
    &\overline{M(E_{\gen}, F_1)} &\qquad &\text{if $5 \mid 2g+8$} \\
    &\overline{M(E_1, F_{\gen})} &\qquad &\text{if $4 \mid g+4$}.
\end{align*}
Note that we are using the irreducibility of $M$ and $CE$ in \cite[Propositions 5.1 and 5.2]{DP} to write these divisors as the closures above.
One readily checks that
\[h^1(E_{\gen} \otimes \det E^\vee_{\gen} \otimes \wedge^2 F_1) = h^1(E_1 \otimes \det E_1^\vee \otimes \wedge^2 F_{\gen}) = 0\]
Hence, when they are defined, $\Hp_{5,g}'$ contains each of 
\[M(E_{\gen}, F_{\gen}), \quad M(E_{\gen}, F_1), \quad M(E_1, F_{\gen}),\]
and all other possible $M(E, F)$ have codimension at least $2$.
\end{proof}

\begin{lem}\label{g3prime}
When $g=3$, the complement of $\Hp'_{5,3}\subset \Hp_{5,3}$ is an irreducible divisor whose class lies in the subgroup generated by $2 a_1$ and $a_2'$.
\end{lem}
\begin{proof}
The divisorial component of the complement of $\Hp_{5,3}'$ inside $\Hp_{5,3}$ is the locus
\[
\overline{M(E_1, F_1)} = \overline{M(\O(1)\oplus\O(2)^{\oplus 3},\O(2)^{\oplus 2}\oplus \O(3)^{\oplus 2}\oplus \O(4))}.
\]
By \cite[Proposition 5.2]{DP}, this locus is precisely the preimage of the hyperelliptic locus under the natural morphism $\Hp_{5,3}\rightarrow \M_3$. By \cite{DL}, the hyperelliptic locus in $\M_3$ has class $9\lambda$. 
The class $\lambda = c_1(f_*\omega_f) = c_1(\pi_* \mathscr{E} \otimes \omega_{\pi}) \in \Pic(\mathscr{H}_{5,3})$ is pulled back from $\Pic(\mathscr{B}_{5,3})$. Since $\Pic(\mathscr{B}_{5,3})$ includes into $\Pic(\mathcal{B}_{5,3})$, we can determine this class via a calculation on the $\SL_2$ quotient, as in Example \ref{exc1}:
\[
[\overline{M(E_1, F_1)}] = 9\lambda=9c_1(f_*\omega_f)=9c_1(\pi_*\E(-2))=54a_1-9a_2' = 27(2 a_1) - 9a_2',
\]
which is in the span of $2a_1$ and $a_2'$.
\end{proof}

\subsection{Excision} 
We proceed similarly to the $k=4$ case. Let $\mathscr{D}_{5,g}\subset \mathscr{X}'_{5,g}\smallsetminus \Hp'_{5,g}$. There is a series of surjections
\begin{equation}
    \Pic(\mathscr{B}_{5,g})\rightarrow \Pic(\mathscr{B}'_{5,g})\cong\Pic(\mathscr{X}'_{5,g})\rightarrow \Pic(\Hp'_{5,g}).
\end{equation}
The middle map is an isomorphism because $\mathscr{X}'_{5,g}\rightarrow \mathscr{B}'_{5,g}$ is a vector bundle. 
We have defined a sequence of morphisms
\begin{equation} \label{scr5}\mathscr{D}_{5,g} \to \mathscr{X}_{5,g}' \rightarrow \mathscr{B}_{5,g}' \rightarrow \mathscr{B}_{5,g} \rightarrow \mathscr{V}_{4,g+4} \times_{\BPGL_2} \mathscr{V}_{5,2g+8} \rightarrow \BPGL_2. 
\end{equation}
\begin{lem}\label{d5}
For $g\geq 2$, some combination of the components of $\mathscr{D}_{5,g}$ has class
\[
\frac{(10g+36)}{\epsilon}(\epsilon a_1)-7a_2'-b_2'\in \Pic(\mathscr{X}'_{5,g})\cong \Pic(\mathscr{B}'_{5,g}).
\]
\end{lem}
\begin{proof}
Base changing \eqref{scr5} by $\BSL_2 \to \BPGL_2$, we obtain the stacks and morphisms considered in \cite[Section 5.3]{part1} (below $\Delta_{5,g}'$ is the complement of the open inclusion $\H_{5,g}' \hookrightarrow \X_{5,g}'$ of \cite[Lemma 5.11]{part1}):
\[\Delta_{5,g}' \to \mathcal{X}_{5,g}' \rightarrow \mathcal{B}_{5,g}' \rightarrow \mathcal{B}_{5,g} \rightarrow \mathcal{V}_{4,g+4} \times_{\BSL_2} \mathcal{V}_{5,2g+8} \rightarrow \BSL_2. \]
By \cite[Lemma 6.6 and Equation 6.19]{part2}, the relation $\frac{(10g+36)}{\epsilon}(\epsilon a_1)-7a_2'-b_2'=0$ holds in $\Pic(\H_{5,g})$, so it must also hold on the open substack $\H'_{5,g}$. Since $\H'_{5,g}=\mathcal{X}'_{5,g}\smallsetminus \Delta_{5,g}'$, any relation among $a_1, a_2', b_2'$ on $\H'_{5,g}$ must come from a class supported on $\Delta_{5,g}'$. Therefore, some combination of components of $\Delta_{5,g}'$ has class $(10g+36)a_1-7a_2'-b_2'$. The corresponding combination of components on $\mathscr{D}_{5,g}$ has the same class.
\end{proof}
\begin{proof}[Proof of Theorem 1.1(3) when $g\geq 4$]
By Lemma \ref{c5}, $\Pic(\Hp_{5,g})\cong \Pic(\Hp_{5,g}')$ for $g\geq 4$. We have that $\Pic(\Hp'_{5,g})$ is a quotient of $\Pic(\mathscr{X}'_{5,g})\cong \Pic(\mathscr{B}'_{5,g})$ by classes supported on $\mathscr{D}_{5,g}$. By Lemma \ref{d5}, one such class is $(10g+36)a_1-7a_2'-b_2'$. Therefore, $\Pic(\Hp'_{5,g})$ is a quotient of 
\[
\frac{\zz (\epsilon a_1) \oplus \zz a_2' \oplus \zz b_2'}{\langle \frac{10g+36}{\epsilon}(\epsilon a_1) - 7a_2' - b_2' \rangle} \cong \zz (\epsilon a_1) \oplus \zz a_2'.
\]
By \cite[Proposition 2.15]{DP}, the rank of $\Pic(\Hp_{5,g})\otimes \qq$ is at least $2$, so we must have that $\Pic(\Hp_{5,g})\cong \zz(\epsilon a_1)\oplus \zz a_2'$.
\end{proof}
\subsubsection{Genus 3}
As in the $k=4$ case, when $g=3$, we require a different argument because Lemma \ref{g3prime} tells us that $\Hp_{5,3}' \subset \Hp_{5,3}$ is the complement of a divisor.  
\begin{proof}[Proof of Theorem \ref{Pic}(2) when $g=3$]
By Lemma \ref{g3prime}, the  kernel of the restriction map \[\Pic(\Hp_{5,3}) \to \Pic(\Hp_{5,3}')\] lies in the subgroup $\langle 2a_1, a_2' \rangle$. We know that $\Pic(\Hp'_{5,g})$ is generated by the classes $a_1, a_2', b_2'$, so it follows that $\Pic(\Hp_{5,3})$ is also generated by these $3$ classes.

Next, we claim that $b_2'$ is integrally expressible in terms of $a_1, a_2'$.
By Lemma \ref{d5}, we know $b_2'$ is expressible in terms of $a_1, a_2'$ in $\Pic(\Hp_{5,3}')$.
But, the kernel of $\Pic(\Hp_{5,3}) \to \Pic(\Hp_{5,3}')$ lies in $\langle a_1, a_2' \rangle$, so $b_2'$ must also be expressible in terms of $a_1, a_2'$ in $\Pic(\Hp_{5,3})$.

It follows that $\Pic(\Hp_{5,3})$ is a quotient of $\zz a_1 \oplus \zz a_2'$.
However, by \cite{DP}, we know that $\Pic(\Hp_{5,3})\otimes \qq$ has rank $2$, so $\Pic(\Hp_{5,3})\cong \zz\oplus \zz$.
\end{proof}
\subsubsection{Genus 2}
As in Section \ref{4genus2}, the argument of Deopurkar--Patel \cite[Proposition 2.15]{DP} in genus $2$ establishes that the rank of $\Pic(\Hp_{5,2})$ is at least $1$.

\begin{proof}[Proof of Theorem \ref{Pic}(2) when $g=2$]
We have already established that $a_1, a_2'$ are generators for $\Pic(\Hp_{5,2})$. We compute directly $\lambda = 5a_1 - a_2'$ as in Section \ref{4genus2}, from which we see $\lambda$ and $a_1$ are generators for  $\Pic(\Hp_{5,2})$.
Arguing as in Section \ref{4genus2}, it suffices to show that the map $\Pic(\M_2) \to \Pic(\Hp_{5,2})$ is injective.

Every degree $5$ line bundle on a genus $2$ curve has a $4$-dimensional space of sections. The Hurwitz space $\Hp_{5,2}$ then sits naturally as an open inside the Grassmann fibration $\mathscr{G} \to \mathscr{P}^5$ parametrizing $2$-dimensional subspaces of the space of global sections of a degree $5$ line bundle.
The complement of $\mathscr{H}_{5,2} \subset \mathscr{G}$ is the locus of pencils with a base point, which we again see is irreducible and not equivalent to the pullback of a divisor on $\mathscr{P}^5$. In particular, the map $\Pic(\mathscr{P}^5) \to \Pic(\Hp_{5,2})$ must be injective. Applying Lemma \ref{hack}, we conclude that $\Pic(\M_2) \to \Pic(\Hp_{5,2})$ is also injective, completing the proof.
\end{proof}

\subsection{Generating line bundles}\label{gb5}
Line bundles generating $\Pic(\Hp_{5,g})$ are given by
\[\mathscr{L}_1 = \begin{cases} \pi_*\left(\det \mathscr{E} \otimes \omega_{\pi}^{\otimes (g+4)/2}\right) & \text{if $g$ even} \\ \\
\pi_*\left((\det \mathscr{E})^{\otimes 2} \otimes \omega_{\pi}^{\otimes (g+4)} \right) & \text{if $g$ odd} \end{cases} \qquad \text{which has} \qquad c_1(\mathscr{L}_1) =  \epsilon a_1\]
and
\[\mathscr{L}_2 = \det f_*(\omega_f) = \det \pi_*(\mathscr{E} \otimes \omega_{\pi}) \qquad\text{which has} \qquad c_1(\mathscr{L}_2) = \lambda = (g+3) a_1 - a_2'.\]

\subsection{Simple branching} Let $T$ and $D$ be as in Figure \ref{TD}. In \cite[Lemma 7.10]{part2}, we wrote the classes of $T$ and $D$ in terms of our generators $a_1$ and $a_2'$:
\[T =(24g + 84)a_1-24a_2' \qquad D = -(32g + 112)a_1 + 36a_2'.\]
Note that the coefficient of $a_1$ is a multiple of $\epsilon$, as it must be because these classes are defined in $\Pic(\Hp_{5,g}) \subseteq \Pic(\H_{5,g})$.
\begin{proof}[Proof of Corollary \ref{mc}(3)]
By excision $\Pic(\mathscr{H}_{5,g}^s) = \Pic(\Hp_{5,g})/\langle T, D \rangle$.
 Again, row operations over $\zz$ diagonalize the change of basis matrix for $\epsilon a_1, a_2'$ to $T, D$:
\[\left(\begin{matrix} 3 & 2 \\
4&3\end{matrix}\right)
\left(\begin{matrix} (24g+84)/\epsilon & -24 \\
(-32g-112)/\epsilon & 36\end{matrix}\right) = 
\left(\begin{matrix} (8g+28)/\epsilon & 0 \\
0 & 12\end{matrix}\right).
\]
For $g \geq 3$, these are the only relations, so $\Pic(\Hp_{5,g}^s)$ is the sum of two cyclic groups of orders equal to the diagonal entries above.

In genus $2$, the above gives $44a_1 = 0$ and $12a_2' =0$, and we have the additional relation 
\[0 = 10\lambda = 10(5a_1 - a_2') \qquad \Rightarrow \qquad 0 = 2(3a_1 + a_2')\]
The generators $a_1$ and $3a_1+a_2'$ generate cyclic groups of order $44$ and $2$ respectively.
\end{proof}

\bibliographystyle{amsplain}
\bibliography{refs}
\end{document}